\documentclass[11pt]{amsart}

\usepackage{amsmath,amssymb,amsthm}
\usepackage{amsmath,amssymb,amsthm,amscd}
\usepackage{enumerate}
\usepackage[frame,cmtip,arrow,matrix,line,graph,curve]{xy}
\usepackage{graphpap, color}
\usepackage[mathscr]{eucal}
\usepackage{color}
\numberwithin{equation}{section}
\newtheorem{prop}{Proposition}
\newtheorem{theo}[prop]{Theorem}
\newtheorem{lemm}[prop]{Lemma}
\newtheorem{coro}[prop]{Corollary}

\theoremstyle{definition}

\theoremstyle{remark}

\newcommand{\p}{\partial}

\begin{document}
\title{An necessary condition for isoperimetric inequality in warped product space}

\author{Chunhe Li}
\address{School of Mathematical Sciences  \\ University of Electronic Science and Technology of China \\ Chengdu, China} \email{chli@fudan.edu.cn}
\author{Zhizhang Wang}
\address{School of Mathematical Sciences\\ Fudan University \\ Shanghai, China}
\email{zzwang@fudan.edu.cn}
\thanks{Research of the first author is supported partially by a NSFC Grant No.11571063, and the last author is supported  by a NSFC Grant No.11301087}
\begin{abstract}
In this note, we show that there is some counterexample  for isoperimetric inequality if the condition $(\phi')^2-\phi''\phi\leq 1$ does not hold  in warped product space.
\end{abstract}

\maketitle

\section{Introduction and preliminary}
The isoperimetric inequality is  one of the central problem in classical differential geometry.  For simple closed curves on plane, it is well known that
$$L^2\geq 4\pi A,$$  where $L$ and $A$ are length of the curve and area  of the domain bounded by curve.   There is a lot of generalization for isoperimetric inequality. (See \cite{Os} for a quick survey.) One aspect is that one can consider these inequalities in general ambient space. The simplest ambient space maybe the space form. For dimension $2$ case, suppose $K$ is its Gauss curvature, then, we have the following type isoperimetric inequality
$$L^2\geq 4\pi A-KA^2.$$
Here the meaning of $L,A$ is same as previous.  The equality in the above will hold, if the curve is a circle.

The more general ambient space may be  warped product space since their geodesic manifold are same. In fact, in $n+1$ dimensional  space $\mathbb{R}^{n+1}$,  the warped product metric is defined by
\begin{eqnarray}\label{metric}
ds^2=\frac{1}{f^2(r)}dr^2+r^2(\sum_{i=1}^{n}\cos^2\theta_{i-1}\cdots\cos^2\theta_1d\theta^2_i).
\end{eqnarray}
Here, $(r,\theta_1,\cdots,\theta_{n})$ is the polar coordinate. We always denote the ambient space with the above metric by $N^{n+1}$. The meaning of the paremeter $r$ is the Euclidean distant of the point $(r,\theta_1,\cdots,\theta_{n})$. Let $S(r)$ be a level set of $r$ and $B(r)$  be the bounded domain enclosed by $S(r)$. We also define the area of $S(r)$ and volume of $B(r)$ by $A(r)$ and $V(r)$. Then we can define some single valuable function $\xi(r)$ by
$$A(r)=\xi(V(r)).$$ Thus, an appropriate isoperimetric inequality in this space has been proved by Guan-Li-Wang \cite{GLW}.

\begin{theo}[Guan-Li-Wang]
Let $\Omega\subset N^{n+1}$ be a domain bounded by a smooth graphical hypersurface $M$ and $S(r_0)$. We assume that $\phi(r)$ and $\tilde{g}$ satisfy condition
$$0\leq (\phi')^2-\phi''\phi\leq 1,$$ then
$$Area(M)\geq \xi(Vol(\Omega)),$$
where $Area(M)$ is the area of $M$ and $Vol(\Omega)$ is the volume of $\Omega$.
\end{theo}
The ambient metric of the above theorem is
\begin{eqnarray}\label{metric1}
ds^2=d\tilde{r}^2+\phi^2(\tilde{r})dS_{n-1}^2.
\end{eqnarray}
\eqref{metric} can be reformulated by the above form. In fact, we let
$$\tilde{r}=\int \frac{dr}{f(r)},\ \  \phi(\tilde{r})=r.$$ Thus, we have
$$1\geq (\frac{d\phi}{d\tilde{r}})^2-\phi\frac{d^2\phi}{d\tilde{r}^2}=f^2-rff',$$ which implies
\begin{eqnarray}\label{cond}
\frac{1-f^2}{r^2}+\frac{ff'}{r}\geq 0.
\end{eqnarray}
In this note, we will show that the above condition is necessary condtion.  Explicitly, we have
\begin{theo}\label{Th2}
If the condition \eqref{cond} does not hold for some $r$,  there is some hypersurface  which is isometric  to the geodesic $r$-sphere and the volume of the convex body bounded by the hypersurface is bigger than the volume of $r$-geodesic ball.
\end{theo}
Using the above theorem, we have the following  corollary for metric \eqref{metric1}. 
\begin{coro}
For the ambient space equipped with metric \eqref{metric1}, if at some $\tilde{r}$, $\frac{d \phi}{d \tilde{r}}\neq 0$, there is some hypersurface  which is isometric  to the geodesic $\phi (\tilde{r})$-sphere and the volume of the convex body bounded by the hypersurface is bigger than the volume of $\phi(\tilde{r})$-geodesic ball.  If at some $\tilde{r}$, $\frac{d \phi}{d \tilde{r}}= 0$, there is some $\tilde{r}'$  closed to $\tilde{r}$ such that we can find some hypersurface  which is isometric  to the geodesic $\phi (\tilde{r}')$-sphere and the volume of the convex body bounded by the hypersurface is bigger than the volume of $\phi(\tilde{r}')$-geodesic ball. In a word, in any case, the condition $ (\phi')^2-\phi''\phi\leq 1$ is necessary for the existence of isoperimetric inequality in warped product space. 
\end{coro}
\begin{proof} 
For the case $\frac{d \phi}{d \tilde{r}}> 0$, by the discussion between \eqref{metric1} and \eqref{cond}, it is obvious from Theorem \ref{Th2} . For the case $\frac{d \phi}{d \tilde{r}}<0$,
we let 
$$r=\phi(\tilde{r}),\ \  f=-\frac{d \phi}{d \tilde{r}},$$ we still can obtain condition \eqref{cond}, which implies that we also have same result for this case. For the case $\frac{d \phi}{d \tilde{r}}=0$, without loss of generality,  we can assume $\phi(\tilde{r})>0$.  Thus the condition $ (\phi')^2-\phi''\phi> 1$ implies $\phi''<0$ at $\tilde{r}$.  Hence, there is some $\tilde{r}'$ close to $\tilde{r}$ such that $\phi'>0$ and  $ (\phi')^2-\phi''\phi> 1$ at $\tilde{r}'$, which is the first case we have discussed. 

\end{proof}
we always use $\cdot$  to present the metric in the ambient space. At first, the following formula  is valid for any $n\geq 1.$

\begin{lemm}\label{openness}
Suppose $\Omega\subset \mathbb{R}^{n+1}$ is an open set. Let $\Sigma$ be $\partial\Omega$ and  $r\frac{\p}{\p r}=(z^1,z^2,\cdots, z^{n+1})$ be its position vector. We also let $X=f(r)r\frac{\p}{\p r}$ be conformal Killing vector. Then the support function of $\Sigma$ is defined by
$$\varphi=X\cdot \nu,$$
where $\nu$ is the outward unit normal on $\Sigma$. Hence
 we have
 \begin{eqnarray}\label{1.1}
 \mathrm{Vol}(\Omega)=\int_{\Sigma}g(r)\varphi dS,
 \end{eqnarray}
 where $dS$ is the volume element of $\Sigma$ and
 \begin{eqnarray}\label{1.2}
 g(r)=\frac{1}{r^{n+1}}\int_{0}^{r}\frac{t^{n}}{f(t)}dt.
 \end{eqnarray}

\end{lemm}
\maketitle
\section{An example}
The aim of this section is to study the isoperimetric inequality using the example obtained in \cite{Li-Wang}.  At first, let's briefly review our example.  For given metric $g$ on $\mathbb{S}^n$, the isometric embedding problem is trying to find some map
$$\vec{r}: \mathbb{S}^n\rightarrow N^{n+1},$$ such that $$d\vec{r}\cdot d\vec{r}=g.$$ The homogenous linearized problem is $$dr\cdot D\tau=0.$$ Here $D$ is the Levi-Civita connection corresponding to $ds^2$. For our case, along the geodesic sphere, the metric does not change. Hence, any fixed vector field is the solution of the above system.

We use $(z^1,z^2,\cdots, z^{n+1})$ as the coordinate of the ambient space $N^{n+1}$. We let $(u_1,u_2,\cdots, u_{n})$ be the sphere coordinate and $(r,\theta_1,\theta_2,\cdots,\theta_{n})$ be the polar coordinate in ambient space.
We present the geodesic sphere of radius $r$ in warped product space by the map
\begin{eqnarray}\label{r}
&&\vec{r}(u_1,\cdots,u_{n})\\
&=&r(\cos u_1\cos u_2\cdots\cos u_{n},\cos u_1\cos u_2\cdots\cos u_{n-1}\sin u_{n},\cdots,\sin u_1).\nonumber
\end{eqnarray}
For any sufficient small constant $\epsilon<\epsilon_0$, we can define some  $C^k$ hypersurface $Y$
\begin{eqnarray}\label{surY}
Y=\vec{r}+(\epsilon+\epsilon^2h^1(\sin u_1)+\epsilon^3h^*(\sin u_1))\frac{\p}{\p z^{n+1}},\nonumber
\end{eqnarray}
where $h^*$ is a single valuable function and $h^1$ is defined by $$h^1(\sin u_1)=\frac{f^2(r)-1}{2rf^2(r)}\sin u_1.$$ We can proved that
$$dY\cdot dY=d\vec{r}\cdot d\vec{r}.$$  More detail of the above discussion can be found in section 8 and section 11 of \cite{Li-Wang}. We will denote the domain bounded by hypersurface $Y$ by $M_{\epsilon}$.

Thus, the hypersurface $Y$ is isometric to the $r$ geodesic sphere. Their area are same. In what follows, we calculate the volume of $M_{\epsilon}$, namely expanding it for parameter $\epsilon$. That is
$$Y*\frac{\p}{\p z^{n+1}}=r\sin u_1+\epsilon+\frac{\alpha}{r}\epsilon^2\sin u_1+o(\epsilon^2).$$
Here we always use $*$ to present the Euclidean inner product in $N^{n+1}$ and we denote $$2\alpha=\frac{f^2(r)-1}{f^2(r)}.$$
Thus, we get $$r^2(\epsilon)=Y*Y=r^2+2\epsilon r\sin u_1+\epsilon^2(1+2\alpha \sin^2 u_1)+o(\epsilon^2).$$ Then, we have
\begin{equation}\label{2.2}
r(\epsilon)=r+\epsilon\sin u_1+\frac{\epsilon^2 }{2r}(\cos^2 u_1+2\alpha \sin^2 u_1) +o(\epsilon^2).
\end{equation}
We also let $2\rho(\epsilon)=r^2(\epsilon)$. By the discussion in paper \cite{Li-Wang}, we know that the support function can be formulated  by
\begin{eqnarray}
\varphi^2&=&2\rho(\epsilon)-\frac{|\nabla \rho(\epsilon)|^2}{f^2(r(\epsilon))}\nonumber\\
&=&r^2+2\epsilon r\sin u_1+\epsilon^2(1+2\alpha \sin^2 u_1)-\frac{\epsilon^2\cos^2u_1}{2f^2(r)}+o(\epsilon^2)\nonumber.
\end{eqnarray}
Here the norm $|\cdot|$ is respect to the canonical metric on $r$ geodesic sphere. Thus, we have
\begin{equation}\label{2.3}
\varphi(\epsilon)=r+\epsilon\sin u_1+\frac{\alpha\epsilon^2}{r}+o(\epsilon^2).
\end{equation}

On the other hand, by \eqref{1.2}, we obviously have
%\begin{eqnarray}
$$\int^r_0\frac{t^{n}}{f(t)}dt=g(r)r^{n+1}.$$ Then we have
\begin{eqnarray}
g'(r)=\frac{1}{rf(r)}-(n+1)\frac{g(r)}{r}
\end{eqnarray}
 and
\begin{eqnarray}\label{2.5}
g''(r)&=&-\frac{f(r)+rf'(r)}{r^2f^2(r)}-(n+1)(\frac{g'(r)}{r}-\frac{g(r)}{r^2})\\
&=&-\frac{f(r)+rf'(r)}{r^2f^2(r)}-(n+1)\frac{1}{r^2f(r)}+(n+1)(n+2)\frac{g(r)}{r^2}.\nonumber
\end{eqnarray}
%Denote $$2\alpha=\frac{f^2(1)-1}{f^2(1)}.$$
Using \eqref{2.2}-\eqref{2.5},  we have
\begin{eqnarray}\label{2.6}
\left.\frac{d(g(r(\epsilon))\varphi(\epsilon))}{d\epsilon}\right|_{\epsilon=0}&=&g'(r)\varphi(0)\left.\frac{dr(\epsilon)}{d\epsilon}\right|_{\epsilon=0}+g(r)\left.\frac{d\varphi(\epsilon)}{d\epsilon}\right|_{\epsilon=0}\\
&=&(\frac{1}{f(r)}-ng(r))\sin u_1,\nonumber
\end{eqnarray}
and
\begin{eqnarray}\label{2.7}
&&\left.\frac{d^2(g(r(\epsilon))\varphi(\epsilon))}{d\epsilon^2}\right|_{\epsilon=0}\\
&=&g''(r)\varphi(0)\left.(\frac{dr(\epsilon)}{d\epsilon})^2\right|_{\epsilon=0}+g'(r)\varphi(0)\left.\frac{d^2r(\epsilon)}{d\epsilon^2}\right|_{\epsilon=0}+2g'(r)\left.\frac{dr(\epsilon)}{d\epsilon}\right|_{\epsilon=0}\left.\frac{d\varphi(\epsilon)}{d\epsilon}\right|_{\epsilon=0}\nonumber\\&&+g(r)\left.\frac{d^2\varphi(\epsilon)}{d\epsilon^2}\right|_{\epsilon=0}\nonumber\\
&=&[-\frac{f(r)+rf'(r)}{rf^2(r)}-\frac{n+1}{rf(r)}+(n+1)(n+2)\frac{g(r)}{r}+2\alpha(\frac{1}{rf(r)}-(n+1)\frac{g(r)}{r})\nonumber\\&&+\frac{2}{rf(r)}-2(n+1)\frac{g(r)}{r}]\sin^2 u_1+[\frac{1}{rf(r)}-(n+1)\frac{g(r)}{r}]\cos^2 u_1+\frac{2\alpha g(r)}{r}.\nonumber
\end{eqnarray}
Hence using Taylor expansion and \eqref{2.6},\eqref{2.7}, we have
\begin{eqnarray}\label{2.8}
&&g(r(\epsilon))\varphi(\epsilon)\\
&=&g(r)r+\epsilon(\frac{1}{f(r)}-ng(r))\sin u_1+\frac{1}{2}\epsilon^2[(-\frac{f'(r)r+nf(r)}{f^{2}(r)r}+\frac{2\alpha}{f(r)r})\sin^2 u_1\nonumber\\&&+(\frac{1}{f(r)r}-\frac{(n+1)g(r)}{r})\cos^2 u_1+\frac{g(r)}{r}((n(n+1)-2\alpha(n+1))\sin^2 u_1+2\alpha )]\nonumber\\&&+o(\epsilon^2).\nonumber
\end{eqnarray}
It is obvious that  $ds^2_n=du_1^2+\cos^2 u_1 ds^2_{n-1},$ where $ds^2_n$  and $ds^2_{n-1}$ are the canonical metric of the $n$ and $n-1$ dimensional unit  sphere $\mathbb{S}^n, \mathbb{S}^{n-1}$. We also denote their volume form by $dS_n$ and $dS_{n-1}$. Thus using the map $\vec{r}$ defined by \eqref{r}, we have
$$\int_{\mathbb{S}^n}\cos^2 u_1dS_n=\int^{\pi/2}_{-\pi/2}\cos^2 u_1 \cos^{n-1} u_1du_1\int_{\mathbb{S}^{n-1}}dS_{n-1}.$$
If we let $t=\sin u_1$, it is clear that
$$\int^{\pi/2}_{-\pi/2}\cos^{n+1} u_1du_1=\int^1_{-1}(1-t^2)^{n/2}dt=n\int^1_{-1}t^2(1-t^2)^{n/2-1}dt,$$
where the last equality comes from integral by parts.
Thus we obtain $$\int^1_{-1}(1-t^2)^{n/2}dt=\frac{n}{n+1}\int_{-1}^1(1-t^2)^{n/2-1}dt,$$ which implies
$$\int_{\mathbb{S}^n} \cos^2 u_1 dS_n=\frac{n}{n+1}\int_{\mathbb{S}^n} dS_n; \ \text{and }  \int_{\mathbb{S}^n} \sin^2u_1 dS_n=\frac{1}{n+1}\int_{\mathbb{S}^n} dS_n.$$
It is also obvious that
$$\int_{\mathbb{S}^n} \sin u_1 dS_n=0.$$

Thus integral both side of \eqref{2.8} on the $r$ geodesic sphere $\mathbb{S}^n(r)$ and use the previous two formulas, then we have
\begin{eqnarray}\label{2.9}
&&\int_{\mathbb{S}^{n}(r)}g(r(\epsilon))\varphi(\epsilon)dS=\int_{\mathbb{S}^{n}}g(r(\epsilon))\varphi(\epsilon)r^ndS_n\\
&=&g(r)r^{n+1}\int_{\mathbb{S}^{n}}dS_n-\frac{r^{n+1}}{f^{3}(r)}(\frac{1-f^{2}(r)}{r^2}+\frac{f(r)f'(r)}{r})\frac{\epsilon^2}{2(n+1)}\int_{\mathbb{S}^n} dS_n+o(\epsilon^2).\nonumber
\end{eqnarray}
Let us denote
$$\Phi(r)=\frac{f(r)f'(r)}{r}+\frac{1-f^2(r)}{r^2}.$$
Then, using \eqref{1.1}, we rewrite \eqref{2.9} to be
$$\text{Vol}({M_{\epsilon}})=\text{Vol}(\mathbb{B}^n(r))-\frac{r^{n+1}\epsilon^2\Phi(r)}{2(n+1)f^{3}(r)}\text{Aera}(\mathbb{S}^n)+o(\epsilon^2).$$
Here $\mathbb{B}^n(r)$  is the $r$ geodesic ball. Hence, if at some $r$, $\Phi(r)<0$, the volume of the perturbated convex body $M_{\epsilon}$ is bigger than the volume of the $r$ geodesic ball. Thus we have proved our main Theorem \ref{Th2}.

At last, let's give some special examples for our theorem.  For space form, $\Phi=0$, the volume is invariant. For ADS space, since 
$$f^2=1-\frac{m}{r}+\kappa r^2,$$ we have $\Phi(r)>0$ which implies  that volume will decrease. If we take
$$f^2=1+\frac{m}{r+1},$$ then, we get
$$\Phi(r)=-\frac{m}{2r(r+1)^2}-\frac{m}{r^2(r+1)}<0,$$ for positive constant $m$. Thus, the volume will increase for perturbation convex body, which implies the isoperimetric inequality will not hold for arbitrary warped product space.

\bigskip
\noindent {\it Acknowledgement:} The authors would like to thank  Professor Pengfei Guan for his interesting and comments.


\begin{thebibliography}{99}
%\bibitem{B} S. Brendle, {\em Constant mean curvature surfaces in warped product space}, Publ. Math. Inst. Hautes etudes Sci.(117) 2013, 247-269.
%\bibitem{HMR} O. Hijazi, S. Montiel, and A. Roldan, {\em Dirac operators on hypersurfaces of manifolds
%with negative scalar curvature,} Ann. Global. Anal. Geom. 23, 247¨C264 (2003)

\bibitem{GLW} Pengfei Guan, Junfang Li and Mu-Tao, Wang, {\em A volume preserving flow and the isoperimetric problem in warped product spaces}, arXiv:1609:08238v1.
\bibitem{Li-Wang}  C. Li and Z. Wang, {\em The Weyl problem in warped product spaces,} arXiv:1603:01350.
%\bibitem{MR}S. Montiel and A. Ros,  {\em Compact hypersurfaces: the Alexandrov theorem for higher
%order mean curvatures,} Differential geometry (ed. by H. Blaine Lawson, Jr., and Keti
%Tenenblat), Pitman Monographs and Surveys in Pure and Applied Mathematics,
%volume 52, pp. 279-C296, Longman Scientific and Technical, 1991
%\bibitem{R} Reilly, R {\em Applications of the Hessian operator in a Riemannian manifold,} Indiana Univ Math J, 26(1977) : 459-472

\bibitem{Os} R. Osserman, {\em The isoperimetric inequality}, Bulletin of the American Mathematical Society, Vol. 84, No. 6, 1978.
\end{thebibliography}
\end{document}